\documentclass[preprint,12pt]{elsarticle}

\usepackage{amsmath, amsthm, amssymb,bbm}

    \usepackage{color}  
     \definecolor{red}{rgb}{0.9,0,0}
     \definecolor{green}{rgb}{0,0.6,0}
     \definecolor{rb}{rgb}{0.6,0,0.2}     
     \definecolor{pass}{rgb}{0,0,0.5}
     \definecolor{blue}{rgb}{0,0,0}
\usepackage{enumitem}

\newcommand{\ts}{\textstyle}

\newcommand{\pt}{\partial}
\newcommand {\eps} {\varepsilon}

\newcommand{\LL}{{\mathcal L}}
\newcommand{\HH}{{\mathcal H}}

\newcommand {\beq} {\begin{equation}}
\newcommand {\eeq} {\end{equation}}
\newcommand {\beqa} {\begin{eqnarray}}
\newcommand {\eeqa} {\end{eqnarray}}
\newcommand {\beqann} {\begin{eqnarray*}}
\newcommand {\eeqann} {\end{eqnarray*}}

\numberwithin{table}{section}

\numberwithin{equation}{section}

\newtheorem{theorem}{Theorem}[section]
\newtheorem{lemma}[theorem]{Lemma}
\theoremstyle{remark}
\newtheorem{remark}[theorem]{Remark}

\journal{J. Comput. Appl. Math.}

\begin{document}

\begin{frontmatter}



\title{How accurate are finite elements on anisotropic triangulations in the maximum norm?\tnoteref{SFI}}
\tnotetext[SFI]{The author wishes to acknowledge financial support
from Science Foundation Ireland Grant SFI/12/IA/1683.}


\author{Natalia Kopteva}

\address{University of Limerick, Limerick, Ireland; {natalia.kopteva@ul.ie}.}

\begin{abstract}
In \cite{Kopt_mc14}
a counterexample of an anisotropic triangulation was given on which
the exact solution has a second-order error of linear interpolation,
while
the computed solution obtained using linear finite elements is only first-order pointwise accurate.
This example was given in the context of a singularly perturbed reaction-diffusion equation.
%
In this paper, we present further examples of unanticipated pointwise convergence behaviour of Lagrange finite elements on anisotropic triangulations.
In particular,
we  show that
linear finite elements may exhibit lower than expected orders of convergence
for the Laplace equation, as well as for certain singular
equations, and their accuracy depends not only on the linear interpolation error, but also on the mesh topology.
%
%
Furthermore, we  demonstrate that 
 pointwise convergence rates which are worse than one might expect
are also observed when higher-order finite elements are employed on anisotropic meshes.
A theoretical justification  will be given for some of the observed numerical phenomena.
\end{abstract}


\end{frontmatter}


\section{Introduction}
\label{sec_intro}
There is a perception in the finite element community, which the author of this article  also shared until recently, that
the finite element solution error in the maximum norm (and, in fact, any other norm) is closely related to the corresponding interpolation error.
It is worth noting that
 an almost best approximation property of finite element solutions in the maximum norm
has been rigourously proved (with a logarithmic factor in the case of linear elements)
for some equations on quasi-uniform meshes \cite{SchWa_best,SchWa_rd}.
To be more precise, with the exact solution $u$
and the corresponding computed solution $u_h$ in a finite element space $S_h$,
one enjoys the error bound
\beq\label{ScjWal}
\|u-u_h\|_{L_\infty(\Omega)}\le  C\ell_h^{\bar r}\,\inf_{\chi\in S_h}\|u-\chi\|_{L_\infty(\Omega)},
\eeq
see \cite{SchWa_best} for the Laplace equation $-\triangle u=f$ with $\ell_h=\ln(1/ h)$,
and
\cite{SchWa_rd} for  singularly perturbed equations of type $-\eps^2\triangle u+u=f$ with $\ell_h=\ln(2+\eps/ h)$;
{\color{blue}see also \cite[Theorem~3.3.7]{ciarlet} and \cite[\S8.6]{BrenScott} for similar maximum norm error bounds}.
Here $h$ is the mesh diameter, $C$ is a positive constant, independent of $h$ and, in the latter case, $\eps$,
while
$\bar r=1$ for linear elements and $\bar r=0$ for higher-order elements in \cite{SchWa_best}, and
$\bar r=1$ in \cite{SchWa_rd}.

Importantly, the proofs of \eqref{ScjWal} were given only for quasi-uniform meshes, while
no such result is known for reasonably general strongly-anisotropic triangulations.
In fact, a counterexample of an anisotropic triangulation  was given in \cite{Kopt_mc14}
on which
the exact solution has a second-order error of linear interpolation,
while
the computed solution obtained using linear finite elements is only first-order pointwise accurate.
Interestingly, it was also shown that, unlike the case of shape-regular meshes,
the convergence rates on anisotropic meshes may depend not only on the linear interpolation error, but also on the mesh topology.
(The latter is also reflected in Fig.\,\ref{fig_higher_order}, left, when $r=1$, {\color{blue}which corresponds to linear finite elements}, and the triangulations of types A and C of Fig.\,\ref{fig_ABC} are used.)

\begin{figure}[b!]
\mbox{\hspace*{-0.9cm}\includegraphics[width=1.05\textwidth]{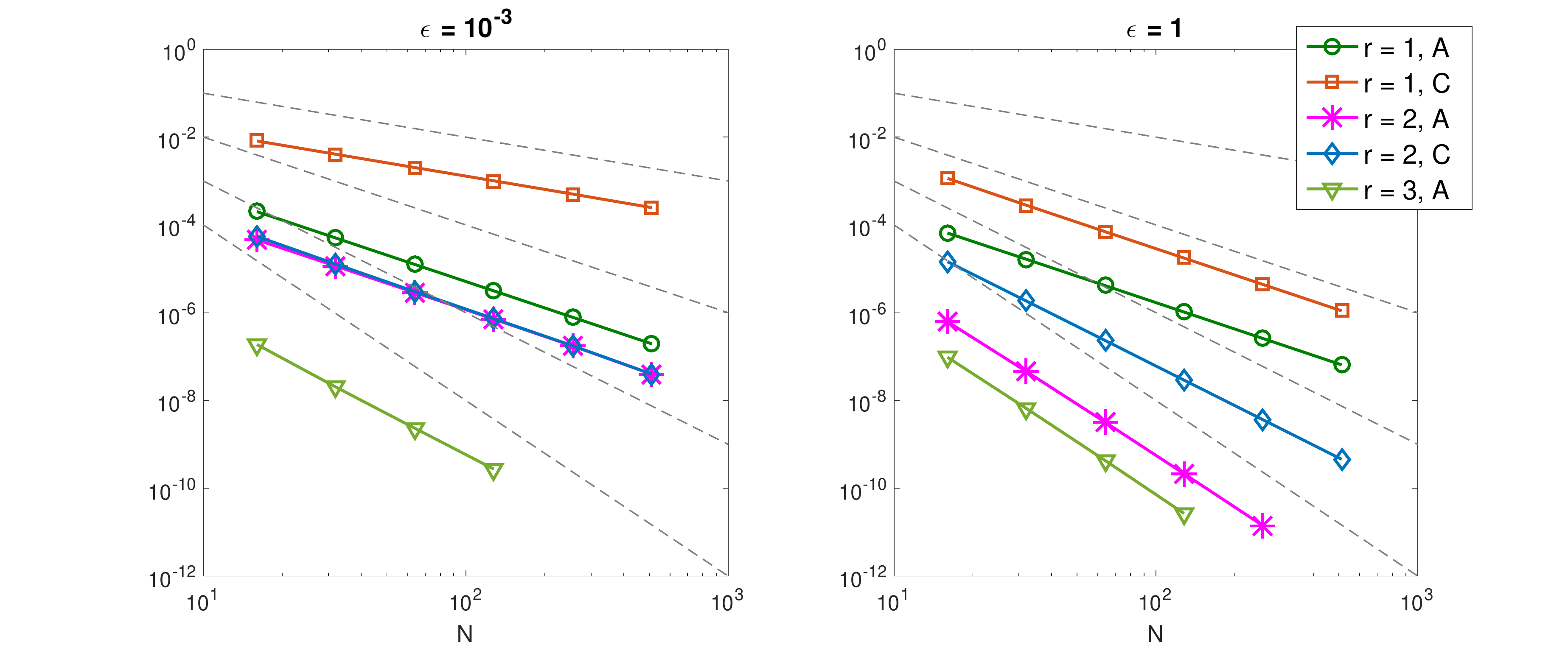}}%
\vspace{-0.3cm}
 \caption{\label{fig_higher_order}
Maximum nodal errors of Lagrange finite elements of degree $r=1,2,3$ on uniform anisotropic triangulations of types A and C
in the domain $(0,2\eps)\times(0,1)$
for $u=e^{-x/\eps}$, $\eps=10^{-3}$ (left) and $\eps=1$ (right); the dashed lines correspond to $N^{-p}$ for $p=1,2,3,4$,
{\color{blue}where $N$ and $\frac14N$  nodes are used in the $x$- and $y$-directions}.}
 \end{figure}

The example in \cite{Kopt_mc14} was given in the context of a singularly perturbed reaction-diffusion equation, and only linear finite elements were considered.
The purpose of  this article is to give further examples of unanticipated convergence behaviour of Lagrange finite element methods on anisotropic triangulations.

\begin{itemize}

\item
We  show that
linear finite elements may be only first-order pointwise accurate also for the Laplace equation, and less than second-order accurate for certain singular equations.

\item
Our findings for the Laplace equation immediately imply that
linear finite elements may be only first-order pointwise accurate when
 applied to an anisotropic diffusion equation on quasi-uniform triangulations.

\item
Effects of the lumped-mass quadrature on the accuracy of linear finite elements are investigated.
It is shown that the lumped-mass quadrature may in certain situations (although not always) improve the orders of convergence from one to two.

\item
We  demonstrate that 
 pointwise convergence rates which are inconsistent with \eqref{ScjWal}, and thus worse than one might expect,
are also observed when higher-order finite elements are employed on anisotropic triangulations.
\end{itemize}

   \begin{figure}[tb]
\hspace*{-0.3cm}
\includegraphics[width=0.36\textwidth]{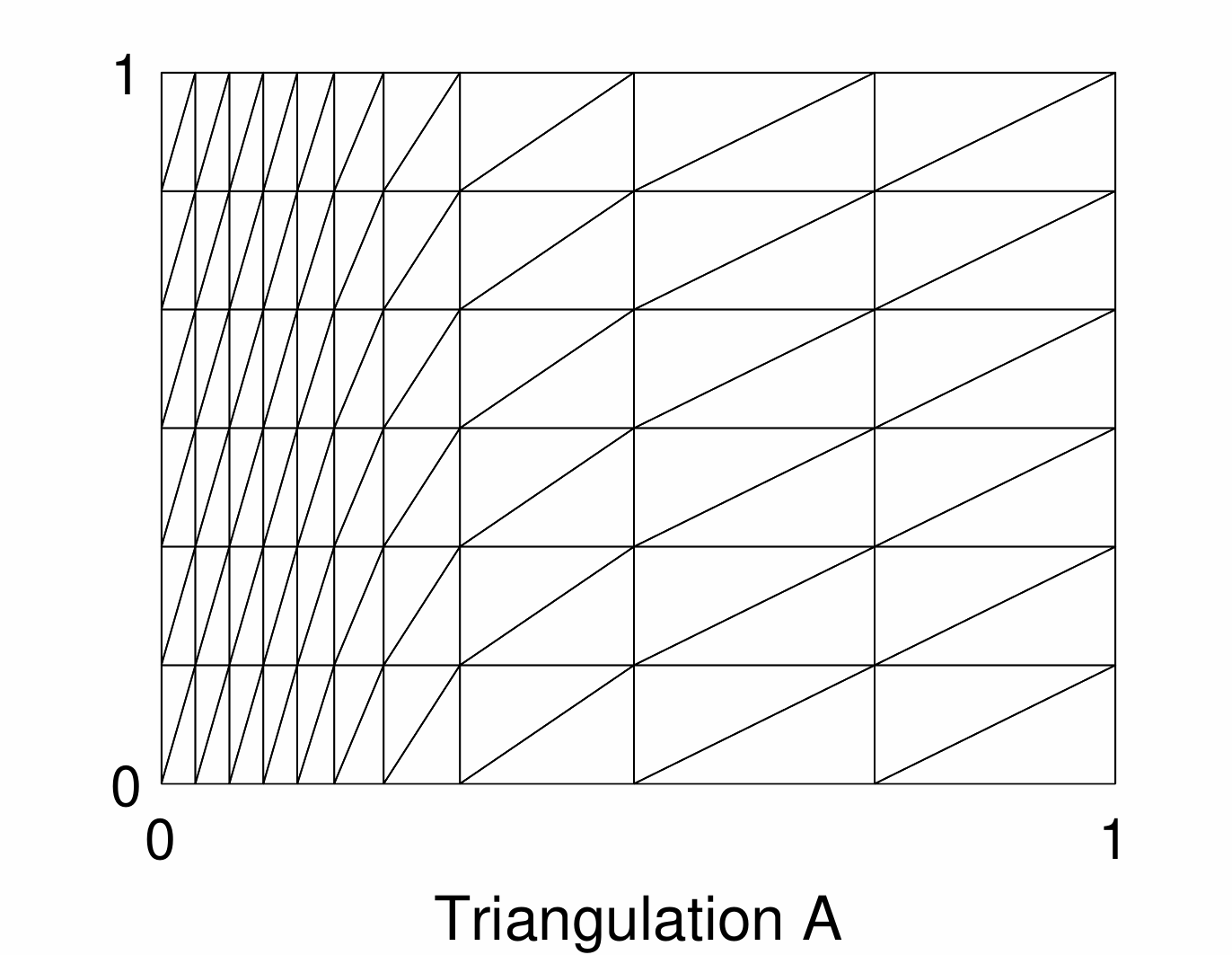}
\hfill
\hspace{-0.7cm}
\includegraphics[width=0.36\textwidth]{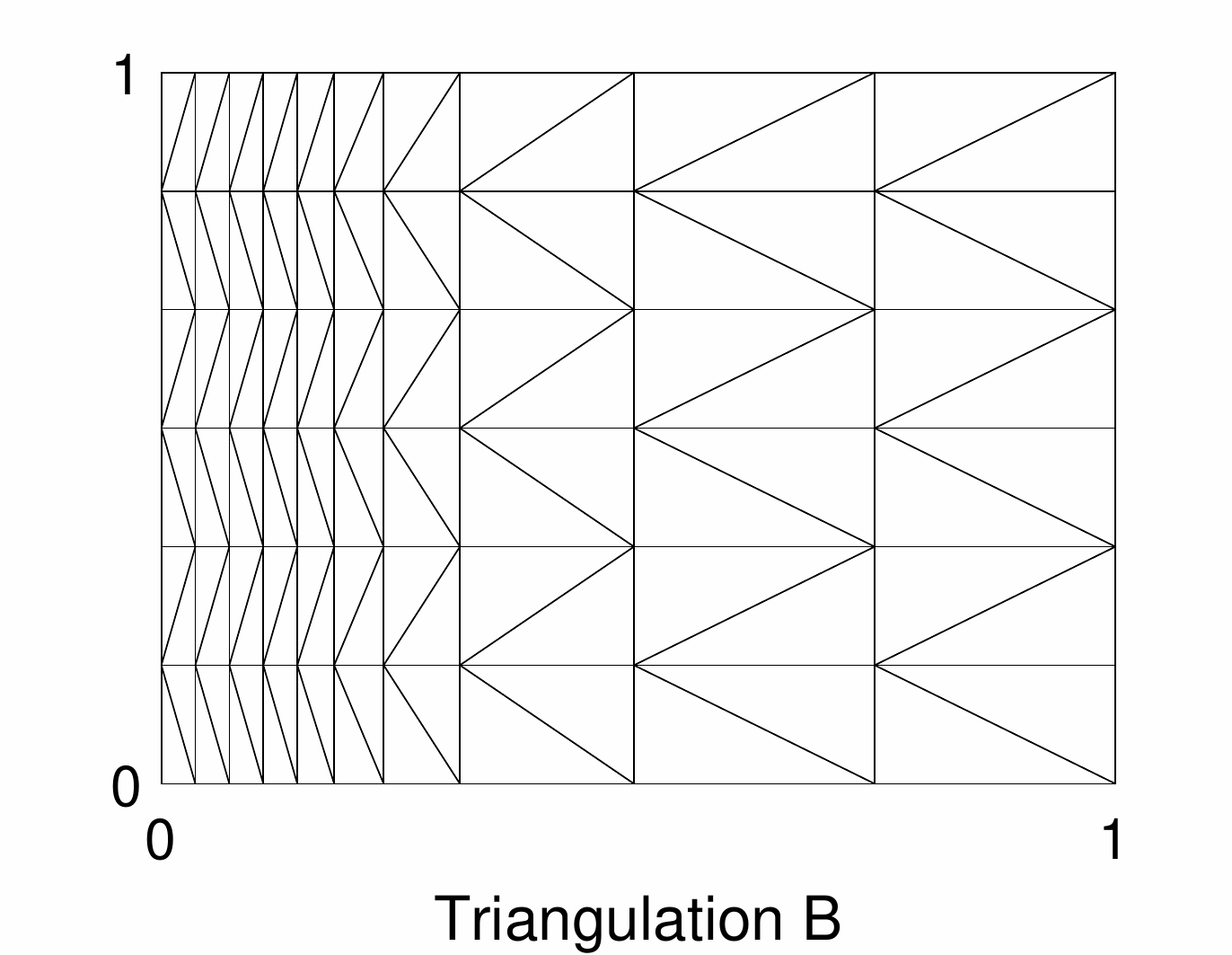}
\hfill
\hspace{-0.7cm}
\includegraphics[width=0.36\textwidth]{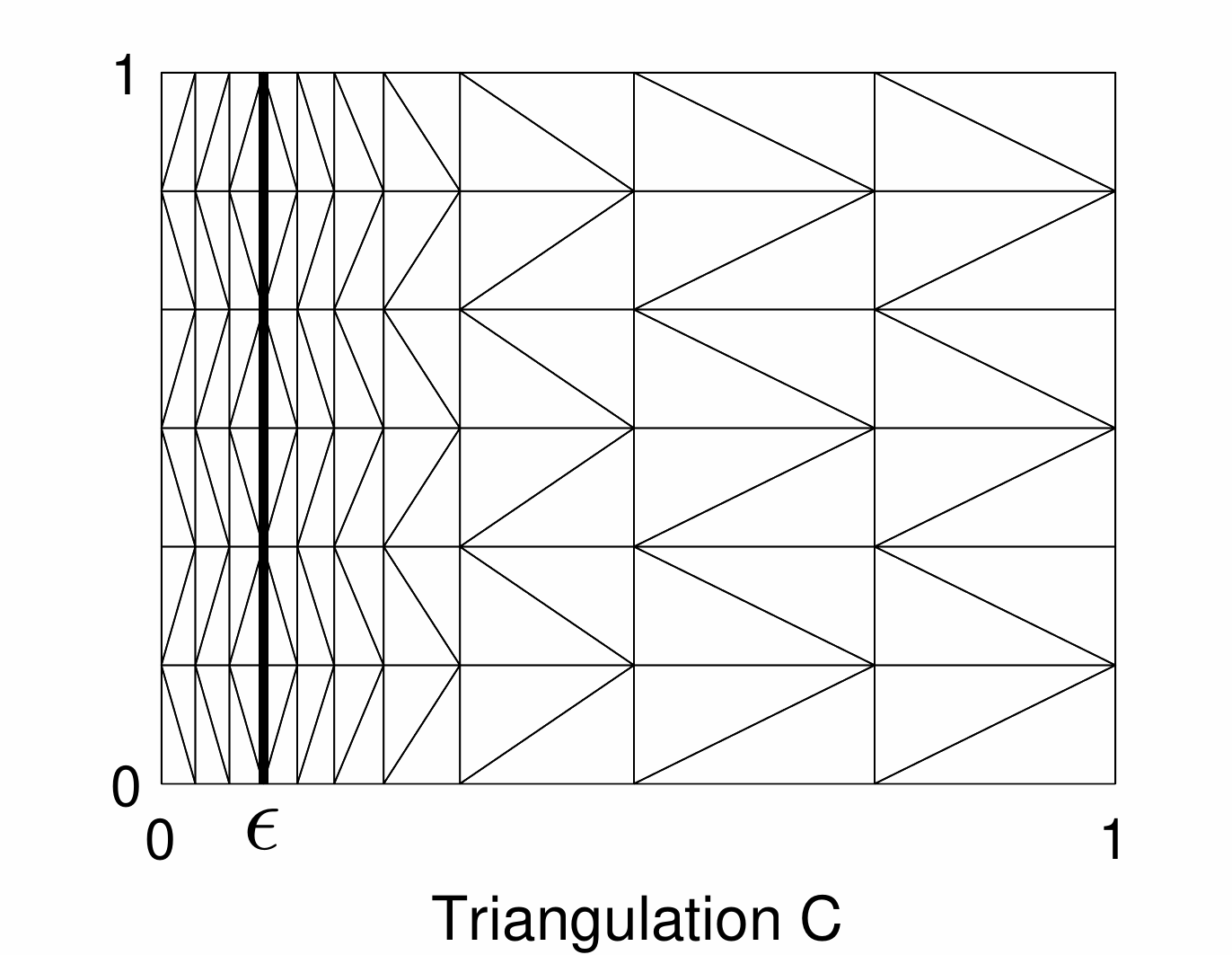}
\hspace{-0.2cm}\vspace{-0.2cm}
 \caption{\label{fig_ABC}
Triangulations 
A, B and C are obtained from the same  rectangular tensor-product grid {\color{blue}(to be specified below)} by
drawing diagonals in a different manner; the number of intervals in each coordinate direction is ${\mathcal O}(N)$.}
 \end{figure}

We shall consider standard Galerkin finite element approximations 
with
Lagrange finite element spaces of fixed degree $r\ge 1$, as well as the lumped-mass version of linear elements. 
As we are interested in counter-examples, our consideration will be restricted to rectangular domains and triangulations obtained from certain
 tensor-product grids by drawing diagonals in each rectangular element (as in Fig.\,\ref{fig_ABC}).
 The underlying tensor-product grids will always have ${\mathcal O}(N)$ intervals in each coordinate direction, and will be
 chosen to ensure that, once the diagonals are drawn in any manner in each rectangular element,
 the interpolation error bound $\|u-u^I\|\le CN^{-(r+1)}$
 holds true for the considered exact solution $u$ and its Lagrange interpolant $u^I\in S_h$ of degree $r$.
For example, when linear finite elements are considered, we have $\|u-u^I\|\le C N^{-2}$, and, furthermore, the considered triangulations are quasi-uniform
{\color{blue}under a Hessian-based metric (with the exception of Shishkin meshes); see \S\ref{sec_hessian}.}
The latter implies that the grids are not made unnecessarily over-anisotropic. It also implies that similar grids may be expected to be obtained using
an anisotropic mesh adaptation.

Some of our findings for Lagrange finite elements of degree $r=1,2,3$ are summarized in Fig.\,\ref{fig_higher_order} (see \S\ref{sec_linear} and \S\ref{sec_higher} for further details).
Here the exact solution is $e^{-x/\eps}$ in the domain $(0,2\eps)\times (0,1)$,
and the triangulations of types A and C of Fig.\,\ref{fig_ABC} are obtained from a uniform rectangular tensor-product grid.
When $\eps=1$ (right), with the meshes remaining quasi-uniform, we observe a textbook behaviour consistent with \eqref{ScjWal}, with the convergence rates close $r+1$.
The only exception is the case of  quadratic elements on triangulation~A, when the nodal superconvergence rates are close to $\mbox{$r+2$}=4$.
By contrast, once one switches to a much smaller $\eps=10^{-3}$ (left), and the triangulations accordingly  become anisotropic, the convergence rates deteriorate to only $r$, with the exception of linear elements
on  triangulation~A.

The paper is organised as follows. The effects of the mesh topology (i.e. of how the diagonals are drawn in rectangular elements)
on the convergence rates
of linear elements are
investigated in \S\ref{sec_linear}. A singularly pertubed equation, the Laplace equation and a singluar equation will be considered.
Next, in \S\ref{sec_quadr},
we address the effects of the lump-mass quadrature on the accuracy of linear finite element solutions.
In \S\ref{sec_higher} quadratic and cubic finite elements are considered.
A theoretical justification for some of the presented numerical phenomena,
which relies on finite difference representations of the considered finite element methods, will be given in \S\ref{sec_theory}.
{\color{blue}Finally, the tested meshes will be discussed in view of Hessian-based metrics in \S\ref{sec_hessian}.}

\section{Linear finite elements: effects of mesh topology}\label{sec_linear}
In this section, triangulations of types shown in Fig.\,\ref{fig_ABC}
in rectangular domains will be compared,
with the underlying rectangular mesh
being
the tensor-product of a  mesh $\{x_i\}_{i=0}^N$ in the $x$-direction and
 the uniform mesh $\{\frac j M\}_{j=0}^M$
 in the $y$-direction, where $M=\frac14N$.
 Two domains will be considered.
  When $\Omega=(0,2\eps)\times(0,1)$, in most cases, we let the mesh $\{x_i\}$ be uniform on $(0,2\eps)$.
 When $\Omega=(0,1)^2$, the mesh $\{x_i\}$ will be a version of the Bakhvalov mesh  \cite{Bak}, which we now describe.
\smallskip

{
\begin{table}[tbp]
\begin{center}
\caption{\label{lin_Bakh_table}Linear elements for equation~\eqref{spp}
on the Bakhvalov tensor-product mesh in $(0,1)^2$  with $M=\frac14 N$: 
maximum nodal errors and
computational rates $p$ in $N^{-p}$}
\begin{tabular}{crcrrrcrrr}
\hline\strut
&&&\multicolumn{3}{l}{\rule{0pt}{12pt}Triangulation A}
&&\multicolumn{3}{l}{\rule{0pt}{12pt}Triangulation C}
\\
&$ N$&&$\varepsilon=1$&
$\varepsilon=2^{-8}$& $\varepsilon= 2^{-16}$
&&$\varepsilon=1$&
$\varepsilon=2^{-8}$& $\varepsilon\le 2^{-16}$
\\\hline
\strut\rule{0pt}{13pt}%
%
%
\strut\rule{0pt}{13pt}%
&  32&&3.56e-6	&5.57e-4	&8.28e-4	&&2.46e-4	&1.28e-2	&1.29e-2	\\	
&&&1.99	&2.12	&2.03	&&2.00	&\bf1.02	&\bf1.01	\\
&  64&&8.94e-7	&1.28e-4	&2.02e-4	&&6.14e-5	&6.34e-3	&6.42e-3	\\	
&&&2.00	&2.02	&2.04	&&2.00	&\bf1.04	&\bf1.01	\\
& 128&&2.24e-7	&3.15e-5	&4.91e-5	&&1.54e-5	&3.08e-3	&3.20e-3	\\	
&&&2.00	&2.06	&2.05	&&2.00	&\bf1.13	&\bf1.00	\\
& 256&&5.60e-8	&7.55e-6	&1.19e-5	&&3.85e-6	&1.40e-3	&1.59e-3	\\	
[-3.3cm]
\rotatebox{90}{\rule{0pt}{9pt}~\,\,No quadrature\,\,~}\\
\hline
\strut\rule{0pt}{13pt}%
&  32&&3.55e-6	&4.58e-4	&7.02e-4	&&2.47e-4	&8.56e-3	&8.58e-3	\\	
&&&1.99	&2.16	&2.04	&&2.00	&\bf0.96	&\bf0.96	\\
&  64&&8.94e-7	&1.03e-4	&1.71e-4	&&6.17e-5	&4.39e-3	&4.41e-3	\\	
&&&2.00	&2.19	&2.04	&&2.00	&\bf1.00	&\bf0.98	\\
& 128&&2.24e-7	&2.25e-5	&4.15e-5	&&1.54e-5	&2.19e-3	&2.23e-3	\\		
&&&2.00	&2.01	&2.05	&&2.00	&\bf1.06	&\bf0.99	\\	
& 256&&5.60e-8	&5.60e-6	&1.00e-5	&&3.85e-6	&1.05e-3	&1.12e-3	\\	
[-3.4cm]
\rotatebox{90}{\rule{0pt}{9pt}~~\,Mass lumping\,~~}\\
\hline
\end{tabular}
\end{center}
\end{table}
}

\noindent
{\it Bakhvalov mesh} \cite{Bak}.
{\color{blue}Set $\sigma:= \eps(r+1)(|\ln\eps|+1)$.
If $\sigma\ge \frac34$, let $\{x_i\}_{i=0}^{N}$ be a uniform mesh on $[0,1]$.
Otherwise, define the mesh $\{x_i\}_{i=0}^{3N/4}$ on $[0,\sigma]$ by
$x_i:=x\bigl([2-\eps]\,{\ts\frac{3i}{4N}}\bigr)$, where
$$
x(t):=
 \eps(r+1)
 \left\{\begin{array}{cl}
t,& t\in[0,1]\\{}
[1-\ln(2-t)],&t\in[1,2-\eps]
 \end{array}\right.
 \!\!.
$$
The remaining part of the mesh $\{x_i\}_{i=3N/4}^{N}$ on $[\sigma,1]$ is uniform.
Note that on $[0,(r+1)\eps]$ this mesh is also uniform, with $x_i-x_{i-1}={\mathcal O}(\eps N^{-1})$,
while on $[(r+1)\eps,\sigma]$ the mesh size is gradually increasing.
Note that a more standard Bakvalov mesh
has a gradually increasing mesh size in the entire layer region. Our version of this mesh
satisfies \eqref{Omega_0}, and so falls within the scope of the theoretical results of \S\ref{sec_hessian}.}

\begin{remark}[Bakhvalov mesh]\label{rem_Bakh}\color{blue}
Suppose that $u=e^{-x/\eps}$ and $r=1$.
Then   a calculation shows that $x'(t)=2\eps$ for $t<1$ and $x'(t)=2\eps/(2-t)$ for $t>1$.
Also, $|u''|^{1/2}=\eps^{-1}e^{-x/(2\eps)}$ so
$|u''(x(t))|^{1/2}\,x'(t)$ is $2e^{-t}$ for $t<1$ and $2e^{-1}$ otherwise.
In other words, $\{x_i\}$ is quasi-uniform under the 1d Hessian metric on $[0,\sigma]$, while $|u''|^{1/2}\le e^{-1}$ for $x\ge\sigma$.
\end{remark}

{
\begin{table}[tb]
\begin{center}
\caption{\label{lin_UNItable}Linear elements for equation~\eqref{spp} on the uniform  mesh in $(0,2\eps)\times(0,1)$,
$M=\frac14 N$:
maximum nodal errors and
computational rates $p$ in $N^{-p}$}
\vspace{0.3cm}
\begin{tabular}{crcrrrcrrr}
\hline\strut
&&&\multicolumn{3}{l}{\rule{0pt}{12pt}Triangulation A}
&&\multicolumn{3}{l}{\rule{0pt}{12pt}Triangulation C}
\\
&$ N$&&$\varepsilon=1$&
$\varepsilon=2^{-8}$& $\varepsilon\le 2^{-16}$
&&$\varepsilon=1$&
$\varepsilon=2^{-8}$& $\varepsilon\le 2^{-16}$
\\\hline
\strut\rule{0pt}{13pt}%
&  32&&	1.65e-5	&5.02e-5	&5.02e-5&&2.81e-4	&4.07e-3	&4.07e-3\\	
&&&1.99	&2.00	&2.00&&	2.00	&\bf 1.01	&\bf1.01\\  	
&  64&&	4.15e-6	&1.25e-5	&1.26e-5&&7.02e-5	&2.02e-3	&2.03e-3\\	
&&&2.00	&2.01	&2.00&&	2.00	&\bf1.02	&\bf1.00\\
& 128&&	1.04e-6	&3.12e-6	&3.14e-6&&1.76e-5	&9.93e-4	&1.01e-3\\	
& &&2.00	&2.02	&2.00&&	2.00	&\bf1.08	&\bf1.00\\
& 256&&	2.59e-7	&7.69e-7	&7.85e-7&&4.41e-6	&4.70e-4	&5.05e-4\\
[-3.3cm]
\rotatebox{90}{\rule{0pt}{9pt}~\,\,No quadrature\,\,~}\\
\hline
\strut\rule{0pt}{13pt}%
&  32&&	1.65e-5	&4.77e-5	&4.77e-5	&&	2.89e-4	&2.98e-3	&2.99e-3\\	
&&&	1.99	&2.00	&2.00		&&2.02	&\bf1.02	&\bf1.02\\  	
&  64&&	4.14e-6	&1.19e-5	&1.19e-5	&&	7.15e-5	&1.47e-3	&1.48e-3\\	
&&&	2.00	&2.00	&2.00		&&2.01	&\bf1.02	&\bf1.01\\
& 128&&	1.04e-6	&2.98e-6	&2.98e-6	&&	1.78e-5	&7.26e-4	&7.34e-4\\	
& &&	2.00	&2.00	&2.00		&&2.00	&\bf1.05	&\bf1.00\\
& 256&&	2.59e-7	&7.46e-7	&7.46e-7	&&	4.43e-6	&3.51e-4	&3.66e-4\\
[-3.4cm]
\rotatebox{90}{\rule{0pt}{9pt}~~\,Mass lumping\,~~}\\
\hline
\end{tabular}
\end{center}
\end{table}
}

\subsection{Singularly perturbed equation}\label{ssec_lin_spp}

We start with the singularly perturbed reaction-diffusion equation considered in \cite{Kopt_mc14}
\beq\label{spp}
-\eps^2\triangle u+u=0,
\eeq
posed in a rectangular domain $\Omega$, subject to Dirichlet boundary conditions,
with the exact solution $u=e^{-x/\eps}$.

The maximum nodal errors of the linear finite element method applied  to this equation are presented in Tables~\ref{lin_Bakh_table} and~\ref{lin_UNItable}.
The Bakhvalov mesh in the domain $(0,1)^2$ and the uniform mesh in the domain $(0,2\eps)\times(0,1)$ were considered.
For the latter case, see also Fig.\,\ref{fig_higher_order}, $r=1$.
Note that here and in all our tables below, the notation $\eps\le 2^{-16}$ is used for $\eps=2^{-k}$ with $k=16,\ldots,24$,
and whenever the computational rates of convergence become substantially lower than $r+1$, they are highlighted in bold face.

We observe that when the triangulations of type~A are used, the convergence rates are consistent with the linear interpolation error, and thus with \eqref{ScjWal}.
Once we switch to the corresponding triangulations of type~C, the convergence rates deteriorate from $r+1=2$ to $1$ for small values of $\eps$.
We conclude that whether linear finite elements are used without quadrature or with lumped-mass quadrature,
their accuracy depends not only on the linear interpolation error, but also on the mesh topology.

{
\begin{table}[tbp]
\begin{center}
\caption{\label{Laplace_table}Laplace equation with the exact solution $u=e^{-x/\eps}$ in the domain $(0,2\eps)\times(0,1)$,
lumped-mass linear elements on the uniform mesh,
$M=\frac14 N$:
maximum nodal errors and
computational rates $p$ in $N^{-p}$
}
\vspace{0.3cm}
\begin{tabular}{rcrrrcrrr}
\hline\strut
&&\multicolumn{3}{l}{\rule{0pt}{12pt}Triangulation A}
&&\multicolumn{3}{l}{\rule{0pt}{12pt}Triangulation C}
\\
$ N$&&$\varepsilon=1$&
$\varepsilon=2^{-8}$& $\varepsilon\le 2^{-16}$
&&$\varepsilon=1$&
$\varepsilon=2^{-8}$& $\varepsilon\le 2^{-16}$
\\\hline
\strut\rule{0pt}{13pt}%
%
  64&&4.45e-6	&1.67e-5	&1.67e-5	&&7.17e-5	&1.92e-3	&1.93e-3	\\	
&&	2.00	&2.00	&2.00	&&2.01	&\bf1.02	&\bf1.01\\
 128&&1.11e-6	&4.17e-6	&4.17e-6	&&1.78e-5	&9.47e-4	&9.62e-4	\\	
&&	2.00	&2.00	&2.00	&&2.00	&\bf1.07	&\bf1.00	\\
 256&&2.79e-7	&1.04e-6	&1.04e-6	&&4.43e-6	&4.52e-4	&4.80e-4	\\
 &&	2.00	&2.00	&2.00	&&2.00	&\bf1.22	&\bf1.00\\
 512&&6.97e-8	&2.61e-7	&2.61e-7	&&1.11e-6	&1.95e-4	&2.40e-4	\\
\hline
\end{tabular}
\end{center}
\end{table}
}

{
\begin{table}[b!]
\begin{center}
\caption{\label{Laplace_table_equi}\color{blue}Laplace equation with the exact solution $u=e^{-x/\eps}$ in the domain $(0,2\eps)\times(0,1)$,
lumped-mass linear elements with $\{x_i\}$ uniform under the 1d Hessian metric:
maximum nodal errors and
computational rates $p$ in $N^{-p}$
}
\vspace{0.3cm}\color{blue}
\begin{tabular}{crcrrrcrrr}
\hline\strut
&&&\multicolumn{3}{l}{\rule{0pt}{12pt}Triangulation A}
&&\multicolumn{3}{l}{\rule{0pt}{12pt}Triangulation C}
\\
&$ N$&&$\varepsilon=1$&
$\varepsilon=2^{-8}$& $\varepsilon\le 2^{-16}$
&&$\varepsilon=1$&
$\varepsilon=2^{-8}$& $\varepsilon\le 2^{-16}$
\\\hline
\strut\rule{0pt}{13pt}%
&  64&&3.70e-6	&1.63e-5	&1.63e-5	&&8.31e-5	&2.17e-3	&2.18e-3	\\	
&& &2.00	&2.00	&2.00	&&1.99	&\bf1.03	&\bf1.01	\\
& 128&&9.26e-7	&4.06e-6	&4.06e-6	&&2.09e-5	&1.06e-3	&1.09e-3	\\	
&&&2.00	&2.00	&2.00	&&2.00	&\bf1.06	&\bf1.00	\\
& 256&&2.31e-7	&1.02e-6	&1.02e-6	&&5.23e-6	&5.09e-4	&5.42e-4	\\	
&&&2.00	&2.00	&2.00	&&2.00	&\bf1.21	&\bf1.00	\\	
& 512&&5.79e-8	&2.54e-7	&2.54e-7	&&1.31e-6	&2.21e-4	&2.71e-4	\\	
[-3.3cm]
\rotatebox{90}{\rule{0pt}{9pt}~~~~~$\,M=\frac14 N\,$~~~~~}\\
\hline
\strut\rule{0pt}{13pt}%
%
&  64&&3.70e-6	&1.63e-5	&1.63e-5	&&8.31e-5	&2.17e-3	&2.18e-3	\\	
&&&2.00	&2.00	&2.00	&&\bf0.92	&\bf1.01	&\bf1.01	\\
& 128&&9.25e-7	&4.06e-6	&4.06e-6	&&4.38e-5	&1.08e-3	&1.09e-3	\\	
&&&2.00	&2.00	&2.00	&&\bf0.98	&\bf1.00	&\bf1.00	\\
& 256&&2.31e-7	&1.02e-6	&1.02e-6	&&2.23e-5	&5.39e-4	&5.42e-4	\\	
&&&2.00	&2.00	&2.00	&&\bf0.99	&\bf1.00	&\bf1.00	\\	
& 512&&5.78e-8	&2.54e-7	&2.54e-7	&&1.12e-5	&2.69e-4	&2.71e-4	\\		
[-3.3cm]
\rotatebox{90}{\rule{0pt}{9pt}~~~~~~$\,M=16\,$~~~~~~}\\
\hline
\end{tabular}
\end{center}
\end{table}
}

\subsection{Laplace equation. Anisotropic diffusion}\label{ssec_Lapl}
We shall now consider the Laplace equation $-\triangle u = f(x,y)$ with the same exact solution $u=e^{-x/\eps}$
subject to the Dirichlet boundary conditions.
This problem will be considered in the domain $(0,2\eps)\times(0,1)$. (When the Bakhvalov mesh is used in $(0,1)^2$, one gets similar numerical results.)

The maximum nodal errors of the lumped-mass linear finite element method applied  to this equation are presented in Table~\ref{Laplace_table}.
Similarly to the case of equation~\eqref{spp},
we again observe that when the triangulations of type~A are used, the convergence rates are consistent with the linear interpolation error, while
on the corresponding triangulations of type~C, the convergence rates deteriorate from $r+1=2$ to $1$ for small values of $\eps$.

Note that the Laplace equation  in the domain $(0,2\eps)\times(0,1)$ can be rewritten as the anisotropic diffusion equation $- \hat u_{xx}-\eps^2 \hat u_{yy} = \hat f(x,y)$
posed in the domain $(0,2)\times (0,1)$ with the exact solution $\hat u(x)=u(\eps x) = e^{-x}$.
Note also that the considered anisotropic uniform triangulations in $(0,2\eps)\times(0,1)$ correspond to quasi-uniform triangulations in $(0,2)\times (0,1)$.
Consequently, our observations for the Laplace equation immediately imply that
linear finite elements may be only first-order pointwise accurate when
 applied to an anisotropic diffusion equation on quasi-uniform triangulations.

A theoretical justification of only first-order accuracy on the considered triangulations of type~C, as well as certain more general locally anisotropic triangulations, will be given in \S\ref{ssec_theory_Lapl}.
\smallskip

\noindent
{\color{blue}{\it Hessian-metric-uniform mesh}.
Furthermore, consider the mesh $\{x_i\}_{i=0}^N$ on $[0,2\eps]$ defined by
$x_i:=x([1-e^{-1}]\frac{i}{N})$, where $x(t):=-2\eps\ln(1-t)$.
Imitating the evaluations in Remark~\ref{rem_Bakh}, one concludes that
$|u''(x(t))|^{1/2}\,x'(t)=2$, so $\{x_i\}$ is uniform under the 1d Hessian metric.
\smallskip

Numerical results for the latter mesh with $M=\frac14 N$ and $M=16$ are given in Table~\ref{Laplace_table_equi}.
In the latter case, the convergence rates are similar to those in Table~\ref{Laplace_table}. However, when $M=16$ is fixed,
which does not affect the interpolation errors, but makes the mesh more anisotropic,
on triangulations of type C
we observe convergence rates close to $1$ even for $\eps=1$.
}

\begin{figure}[b!]
~\hfill\mbox{\hspace*{-0.0cm}\includegraphics[width=0.9\textwidth]{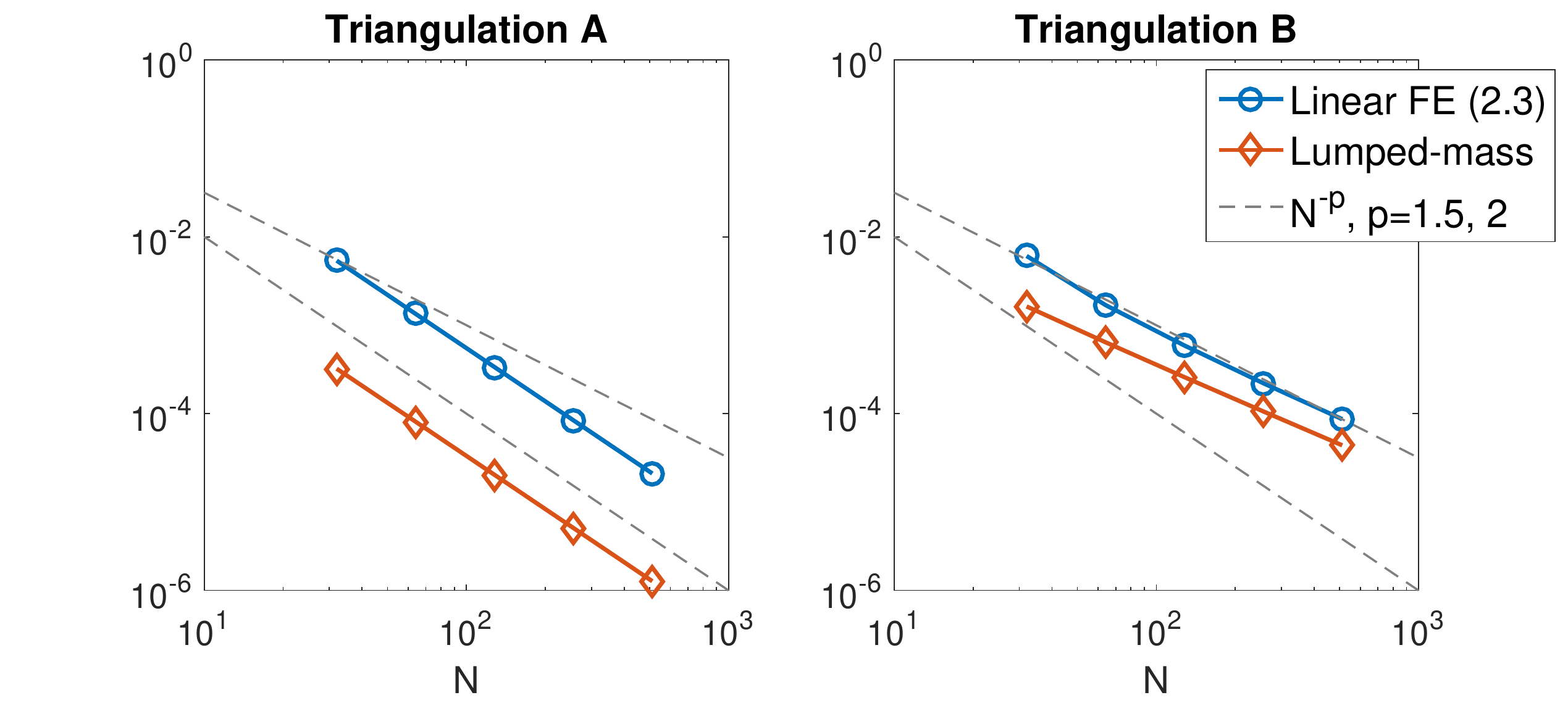}}\hfill~%
\vspace{-0.3cm}
 \caption{\label{fig_singular}
Singular equation \eqref{sing}: maximum nodal errors  for linear finite elements on graded anisotropic triangulations A and B.}
 \end{figure}

\subsection{Singular equations on graded anisotropic meshes}\label{ssec_singular}
{\color{blue}Singular equations also exhibit layer solutions, so
it is quite appropriate to employ anisotropic meshes in their numerical solution.
}
Consider a  version of the test problem from \cite{NSSV06}
\beq\label{sing}
-\triangle u+ f(u)=0,\qquad f(u):=-{\textstyle \frac14}u^{-3},
\eeq
posed in the square domain $(0,1)^2$, with the exact solution $u=x^{1/2}$, subject to the appropriate Dirichlet boundary conditions.
Following \cite{NSSV06}, $f$ will be replaced in our computations by $\tilde f(u):=-{\textstyle \frac14}\max\{u,\mu\}^{-3}$ with some small positive $\mu$.
Note (see, e.g.,\cite[(5.1)]{NSSV06}) that the corresponding exact solution $\tilde u$ satisfies $0\le u-\tilde u\le \mu$.
{\color{blue}(Note also that there exist unique $u$ and $\tilde u$ in $C(\bar\Omega)$ by \cite[Lemma~1]{DK14}; see also references in \cite[\S2]{NSSV06}.)}

For this problem we shall consider the linear finite element method in the form
\beq
( \nabla u_h ,  \nabla \chi) + (\tilde f(u_h)^I, \chi ) = 0\qquad \forall\,\chi \in S_h,
\label{Noch_fem}
\eeq
where $\tilde f(u_h)^I$ is the standard piecewise-linear Lagrange interpolant of $\tilde f(u_h)$.
The lumped-mass version of this method will also be considered, which is obtained from \eqref{Noch_fem}
by replacing $(\tilde f(u_h)^I, \chi )$ with $\int_{\Omega} \bigl(\tilde f(u_h)\chi\bigr)^I$.
The discrete nonlinear problems are solved using the damped Newton method.

To satisfy $|u-u^I|\le C N^{-2}$, the graded mesh $x_i=(i/N)^4$ is employed. 
We set $\mu=N^{-2}$ in all our computations.
Note that if $\mu$ is chosen too small, it appears that the errors for the method \eqref{Noch_fem}
become dominated by ${\mathcal O}(h_1^2/\mu^3)$, where $h_1=x_1-x_0$, while our choice $\mu=N^{-2}$ does not affect the rates of convergence.
Note also that for the lumped-mass version there are no such restrictions on $\mu$, and one can use, for example $\color{blue}\mu={\mathcal O}(N^{-3})$.

Triangulations of types A and B were considered, with the maximal nodal errors shown in Fig.\,\ref{fig_singular}.
We again observe that the computational rates of convergence for the triangulations of type A are consistent with the linear interpolation errors, and thus with
\eqref{ScjWal}, while once we switch to the triangulations of type~B, the convergence rates become lower that $3/2$.
In other words,
the convergence rates on the considered graded anisotropic meshes depend not only on the linear interpolation error, but also, and very considerably, on the mesh topology.

\begin{remark}[Convergence in the $L_2$ norm]\color{blue}
Unsurprisingly, the computational  rates of convergence on triangulations of type B in the $L_2$ norm are closer to the optimal~$2$. To be more precise, they are close to $1.8$
in the lumped-mass case, and to $1.9$ when no quadrature is used (while for triangulations of type A we observed convergence rates close to $2$, similarly to those in the maximum norm).
\end{remark}

\section{Linear finite elements: effects of lumped-mass quadrature}\label{sec_quadr}

   \begin{figure}[b!]
\mbox{\hspace*{-1cm}\includegraphics[width=1.05\textwidth]{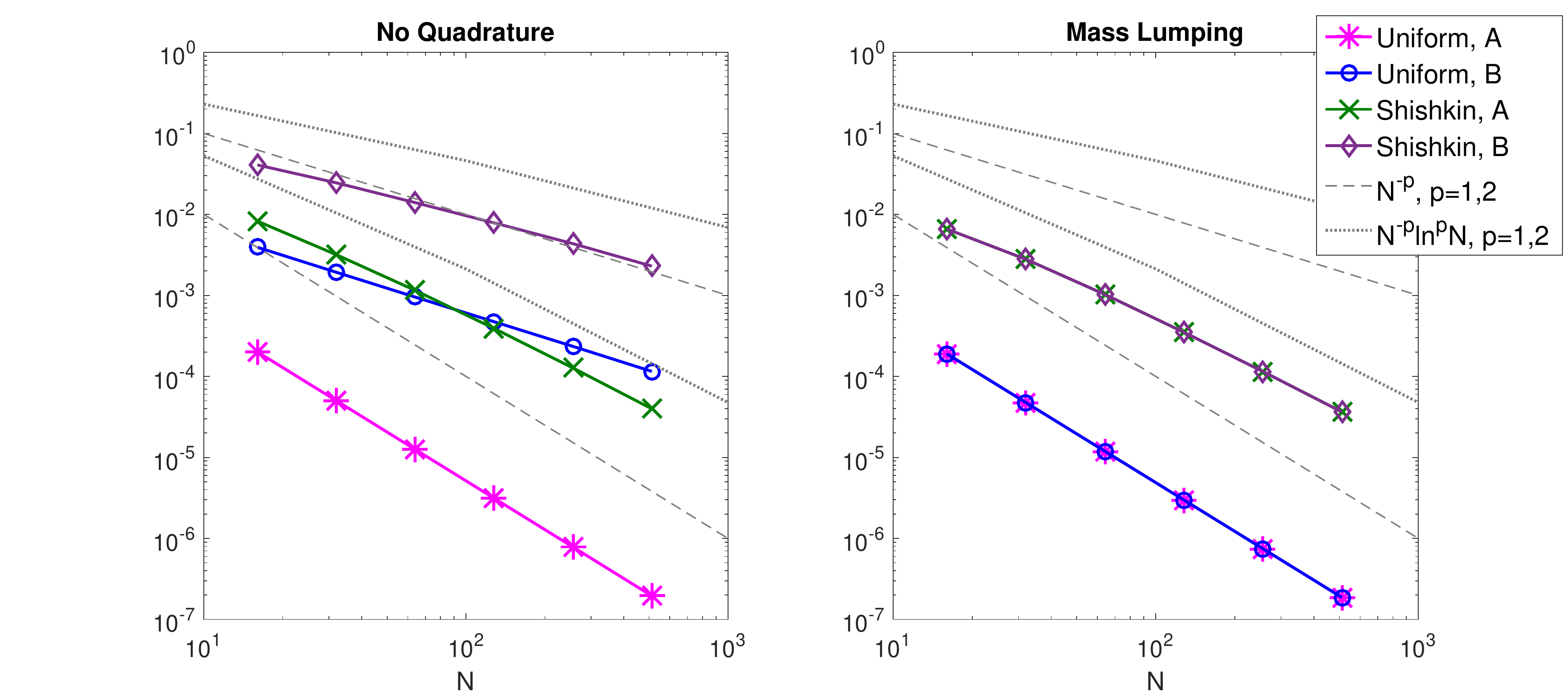}}%
\vspace{-0.3cm}
 \caption{\label{fig_quadr}
Maximum nodal errors of linear finite elements for equation~\eqref{spp} on triangulations A and B for $\eps=10^{-3}$: uniform mesh $(0,2\eps)\times(0,1)$ and Shishskin mesh in $(0,1)^2$.}
 \end{figure}

In this section,
we are interested in the
the effects of the lumped-mass quadrature on the accuracy of linear finite elements.
It will be shown that the lumped-mass quadrature may in certain situations (although not always, as was demonstrated in~\S\ref{sec_linear}) improve the orders of convergence from $1$ to~$2$.

As a test problem we again use the singularly perturbed equation \eqref{spp} with the exact solution $u=e^{-x/\eps}$
in the rectangular domains $(0,2\eps)\times(0,1)$ and $(0,1)^2$.
Triangulations of types A and B of Fig.\,\ref{fig_ABC} will be considered, with the underlying grid being uniform in the domain $(0,2\eps)\times(0,1)$,
and the tensor-product of a Shishkin-type grid in the $x$-direction and the uniform grid in the $y$-direction  in the domain $(0,1)^2$.
Piecewise-uniform layer-adapted Shishkin grids are frequently employed in the numerical solution of equations

{
\begin{table}[t!]
\begin{center}
\caption{\label{lin_quadrature_UNItable}Linear elements for equation~\eqref{spp} on the uniform  mesh in $(0,2\eps)\times(0,1)$,
$M=\frac14 N$:
maximum nodal errors and
computational rates $p$ in $N^{-p}$
}
\vspace{0.3cm}
\begin{tabular}{rcrrrcrrr}
\hline\strut
&&\multicolumn{3}{l}{\rule{0pt}{12pt}Triangulation B}
\\
&&\multicolumn{3}{l}{\rule{0pt}{12pt}No quadrature}
&&\multicolumn{3}{l}{\rule{0pt}{12pt}Mass lumping}
\\
$ N$&&$\varepsilon=1$&
$\varepsilon=2^{-8}$& $\varepsilon\le 2^{-16}$
&&$\varepsilon=1$&
$\varepsilon=2^{-8}$& $\varepsilon\le 2^{-16}$
\\\hline
\strut\rule{0pt}{13pt}%
  32&&4.44e-5	&1.93e-3	&1.93e-3	&&	1.65e-5	&4.77e-5	&4.77e-5	\\	
 &&	2.74	&\bf 1.02	&\bf 1.02 &&	1.99	&2.00	&2.00\\
  64&&6.65e-6	&9.49e-4	&9.54e-4	&&	4.14e-6	&1.19e-5	&1.19e-5	\\	
&&	2.45	&\bf 1.03	&\bf 1.01 &&	2.00	&2.00	&2.00\\	
 128&&1.22e-6	&4.64e-4	&4.74e-4	&&	1.04e-6	&2.98e-6	&2.98e-6	\\
&&2.12	&\bf 1.09	&\bf 1.00 &&	2.00	&2.00	&2.00\\ 	
 256&&2.79e-7	&2.17e-4	&2.36e-4	&&	2.59e-7	&7.46e-7	&7.46e-7	\\
%
%
\hline
\end{tabular}
\end{center}
\end{table}
}

{
\begin{table}[h!]
\begin{center}
\caption{\label{lin_quadrature_Shi_table}Linear elements for equation~\eqref{spp} on the Shishkin tensor-product mesh in $(0,1)^2$  with $M=\frac14 N$:
maximum nodal errors and
computational rates $p$ in $(N^{-1}\ln N)^p$
}
\vspace{0.3cm}
\begin{tabular}{rcrrrcrrr}
\hline\strut
&&\multicolumn{3}{l}{\rule{0pt}{12pt}Triangulation B}
\\
&&\multicolumn{3}{l}{\rule{0pt}{12pt}No quadrature}
&&\multicolumn{3}{l}{\rule{0pt}{12pt}Mass lumping}
\\
$ N$&&$\varepsilon=1$&
$\varepsilon=2^{-8}$& $\varepsilon\le 2^{-16}$
&&$\varepsilon=1$&
$\varepsilon=2^{-8}$& $\varepsilon\le 2^{-16}$
\\\hline
\strut\rule{0pt}{13pt}%
  32&&2.38e-5	&2.45e-2	&2.46e-2	&&3.55e-6	&2.80e-3	&2.80e-3	\\	
 &&	3.74	&\bf1.12	&\bf1.10	&&2.70	&1.96	&1.96\\
  64&&3.52e-6	&1.38e-2	&1.40e-2	&&8.94e-7	&1.03e-3	&1.03e-3	\\	
&&	3.65	&\bf1.15	&\bf1.06	&&2.57	&1.99	&1.99\\	
 128&&4.92e-7	&7.41e-3	&7.90e-3	&&2.24e-7	&3.51e-4	&3.51e-4	\\
&&3.43	&\bf1.32	&\bf1.03	&&2.48	&2.00	&2.00\\ 	
 256&&7.20e-8	&3.54e-3	&4.43e-3	&&5.60e-8	&1.15e-4	&1.15e-4	\\
%
%
\hline
\end{tabular}
\end{center}
\end{table}
}

\noindent of type~\eqref{spp}
\cite{Shi,Ko_OR}.
The Shishkin grid that we consider has the transition parameter
$\sigma = \min\bigl\{2\eps\ln N,\, \frac12\bigr\}$ and
equal numbers of grid points on $(0,\sigma)$ and $(\sigma,1)$,
so a calculation shows that the interpolation error $|u-u^I|\le C N^{-2}\ln^2 N$.

The maximum nodal errors are given in Table~\ref{lin_quadrature_UNItable} for the uniform mesh and in Table~\ref{lin_quadrature_Shi_table} for the Shishkin mesh;
for $\eps=10^{-3}$ see also Fig.\,\ref{fig_quadr}.
We observe that while the triangulations of type~A are used, whether no quadrature or the mass-lumped quadrature is employed, the errors are consistent with
the linear interpolation errors, and thus with~\eqref{ScjWal}.
Once we switch to the triangulations of type B, we observe a similar convergence behaviour for the lumped-mass quadrature version, however
the rates of convergence deteriorate to $1$ (with the logarithmic factor in the case of the Shishkin mesh) if no quadrature is used.
Note that here the situation is entirely different from what was observed in \S\ref{ssec_lin_spp}, where all numerical results were qualitatively independent of the quadrature.

A theoretical justification of lower-order accuracy on the considered triangulations of type~B, as well as certain more general locally anisotropic triangulations, will be given in
 \S\ref{ssec_theory_quad}.

\section{Higher-order elements on anisotropic triangluations}\label{sec_higher}

{
\begin{table}[t!]
\begin{center}
\caption{\label{quadr_Bakh_table}Quadratic elements for equation~\eqref{spp} on the Bakhvalov tensor-product mesh in $(0,1)^2$, 
$M=\frac14 N$:
maximum nodal errors and
computational rates $p$ in $N^{-p}$
}
\vspace{0.3cm}
\begin{tabular}{rcrrrcrrr}
\hline\strut
&&\multicolumn{3}{l}{\rule{0pt}{12pt}Triangulation A}
&&\multicolumn{3}{l}{\rule{0pt}{12pt}Triangulation C}
\\
$ N$&&$\varepsilon=1$&
$\varepsilon=2^{-8}$& $\varepsilon\le 2^{-16}$
&&$\varepsilon=1$&
$\varepsilon=2^{-8}$& $\varepsilon\le 2^{-16}$
\\\hline
\strut\rule{0pt}{13pt}%
%
%
  32&& 1.12e-8	&2.80e-4	&2.88e-4 &&2.96e-7	&3.35e-4	&3.42e-4	\\
&&3.84	&\bf2.10	&\bf2.00&&3.00	&\bf2.20	&\bf2.11	\\	
  64&& 7.82e-10	&6.56e-5	&7.19e-5 &&3.69e-8	&7.31e-5	&7.92e-5	\\
&&3.89	&\bf2.29	&\bf2.00&&3.00	&\bf2.34	&\bf2.07	\\  	
 128&& 5.28e-11	&1.34e-5	&1.80e-5 &&4.62e-9	&1.45e-5	&1.89e-5	\\
&&3.92	&2.67	&\bf2.00&&3.00	&2.68	    &\bf2.04	\\ 	
 256&& 3.48e-12	&2.11e-6	&4.49e-6&&5.77e-10	&2.26e-6	&4.60e-6	\\	
\hline
\end{tabular}
\end{center}
\end{table}
}

{
\begin{table}[t!]
\begin{center}
\caption{\label{quadr_UNItable}Quadratic elements for equation~\eqref{spp} on the uniform  mesh in $(0,2\eps)\times(0,1)$,
$M=\frac14 N$:
maximum nodal errors and
computational rates $p$ in $N^{-p}$
}
\vspace{0.3cm}
\begin{tabular}{rcrrrcrrr}
\hline\strut
&&\multicolumn{3}{l}{\rule{0pt}{12pt}Triangulation A}
&&\multicolumn{3}{l}{\rule{0pt}{12pt}Triangulation C}
\\
$ N$&&$\varepsilon=1$&
$\varepsilon=2^{-8}$& $\varepsilon\le 2^{-16}$
&&$\varepsilon=1$&
$\varepsilon=2^{-8}$& $\varepsilon\le 2^{-16}$
\\\hline
\strut\rule{0pt}{13pt}%
  32&&4.59e-8	&1.14e-5	&1.15e-5	   && 1.85e-6	&1.27e-5	&1.28e-5	\\	
 &&3.84	&\bf2.03	&\bf2.00	   &&2.99	&\bf2.11	&\bf2.07	\\
  64&&3.20e-9	&2.79e-6	&2.88e-6	   && 2.33e-7	&2.94e-6	&3.03e-6	\\	
&&3.89	&\bf2.13	&\bf2.00	   &&2.99	&\bf2.16	&\bf2.04	\\	
 128&&2.15e-10	&6.37e-7	&7.19e-7	   && 2.92e-8	&6.57e-7	&7.39e-7	\\
&&3.93	&\bf2.42	&\bf2.00	   &&3.00	&\bf2.43	&\bf2.02	\\ 	
 256&&1.42e-11	&1.19e-7	&1.80e-7	   && 3.66e-9	&1.22e-7	&1.82e-7	\\
%
%
\hline
\end{tabular}
\end{center}
\end{table}
}

Next, we shall numerically investigate quadratic and cubic Lagrange finite elements applied to
the singularly perturbed equation~\eqref{spp} with the  exact solution $u=e^{-x/\eps}$,
posed in the domains $(0,1)^2$ and $(0,2\eps)\times(0,1)$.
In these domains we
 respectively consider the Bakhvalov and the uniform rectangular grids  described in \S\ref{sec_linear}.
This ensures  that the interpolation error is ${\mathcal O}(N^{r+1})$, while the maximum nodal errors are
shown in Fig.\,\ref{fig_higher_order}, see $r=2,\,3$, and presented in Tables~\ref{quadr_Bakh_table} and~\ref{quadr_UNItable}
for $r=2$, and in Table~\ref{cubic_UNItable} for $r=3$.

When $\eps=1$, the triangulations remain quasi-uniform, and we observe a textbook behaviour consistent with \eqref{ScjWal}, with the convergence rates close $r+1$.
The only exception is the case of  quadratic elements on triangulation~A, when the nodal superconvergence rates become close to $r+2=4$.
However, as $\eps$ takes considerably smaller values, and the triangulations accordingly  become anisotropic,
the convergence behaviour changes quite dramatically, with
the convergence rates deteriorating to only $r$.
Note that in contrast to linear elements, the mesh topology
(i.e. whether triangulations of type A or C are used)
does not seem to affect the rates of convergence, which consistently remain lower than one may conjecture
motivated by~\eqref{ScjWal}.


{
\begin{table}[t!]
\begin{center}
\caption{\label{cubic_UNItable}Cubic elements for equation~\eqref{spp} on the uniform  mesh in $(0,2\eps)\times(0,1)$,
$M=\frac14 N$:
maximum nodal errors and
computational rates $p$ in $N^{-p}$
}
\vspace{0.3cm}
\begin{tabular}{rcrrrr}
\hline\strut
&&\multicolumn{4}{l}{\rule{0pt}{12pt}Triangulation A}
\\
$ N$&&$\varepsilon=1$&
$\varepsilon=2^{-8}$& $\varepsilon=2^{-16}$&$\varepsilon=2^{-24}$\\\hline
\strut\rule{0pt}{13pt}%
  16&& 9.68e-08 &  1.89e-07 &  1.89e-07  & 1.89e-07\\
&& 3.87 & \bf3.23 & \bf3.22 & \bf3.22\\
  32&& 6.60e-09 &  2.01e-08 &  2.02e-08  & 2.02e-08\\
&& 3.94 & \bf3.15 & \bf3.12 & \bf3.12\\
  64&& 4.31e-10 &  2.25e-09 &  2.33e-09  & 2.33e-09\\
&& 4.01 & \bf3.25 & \bf3.07 & \bf3.06\\
 128&& 2.68e-11 &  2.37e-10 &  2.77e-10  & 2.80e-10\\
%
%
\hline
\end{tabular}
\end{center}
\end{table}
}

\section{Theoretical justification}\label{sec_theory}

In this section, we give a theoretical justification for some of
the numerical phenomena presented in the previous sections.
For this purpose, it is convenient to look at the considered finite element methods as certain finite difference schemes
on the underlying rectangular
 tensor-product meshes.

     \begin{figure}[t!]
\hspace{0.5cm}\includegraphics[height=0.35\textwidth,angle=180]{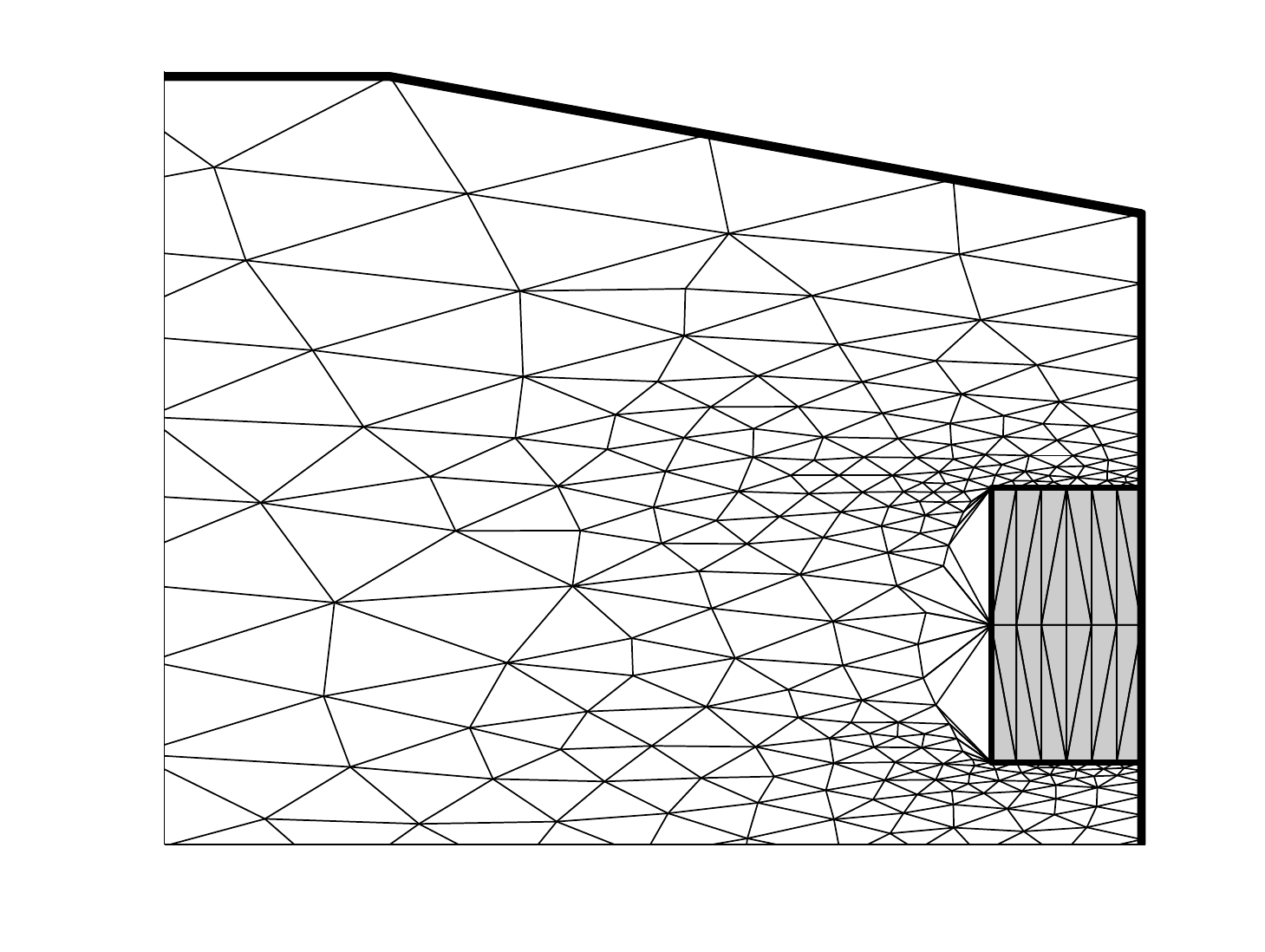}
\hfill{}\vspace{-0.32\textwidth}

\hfill\includegraphics[height=0.31\textwidth]{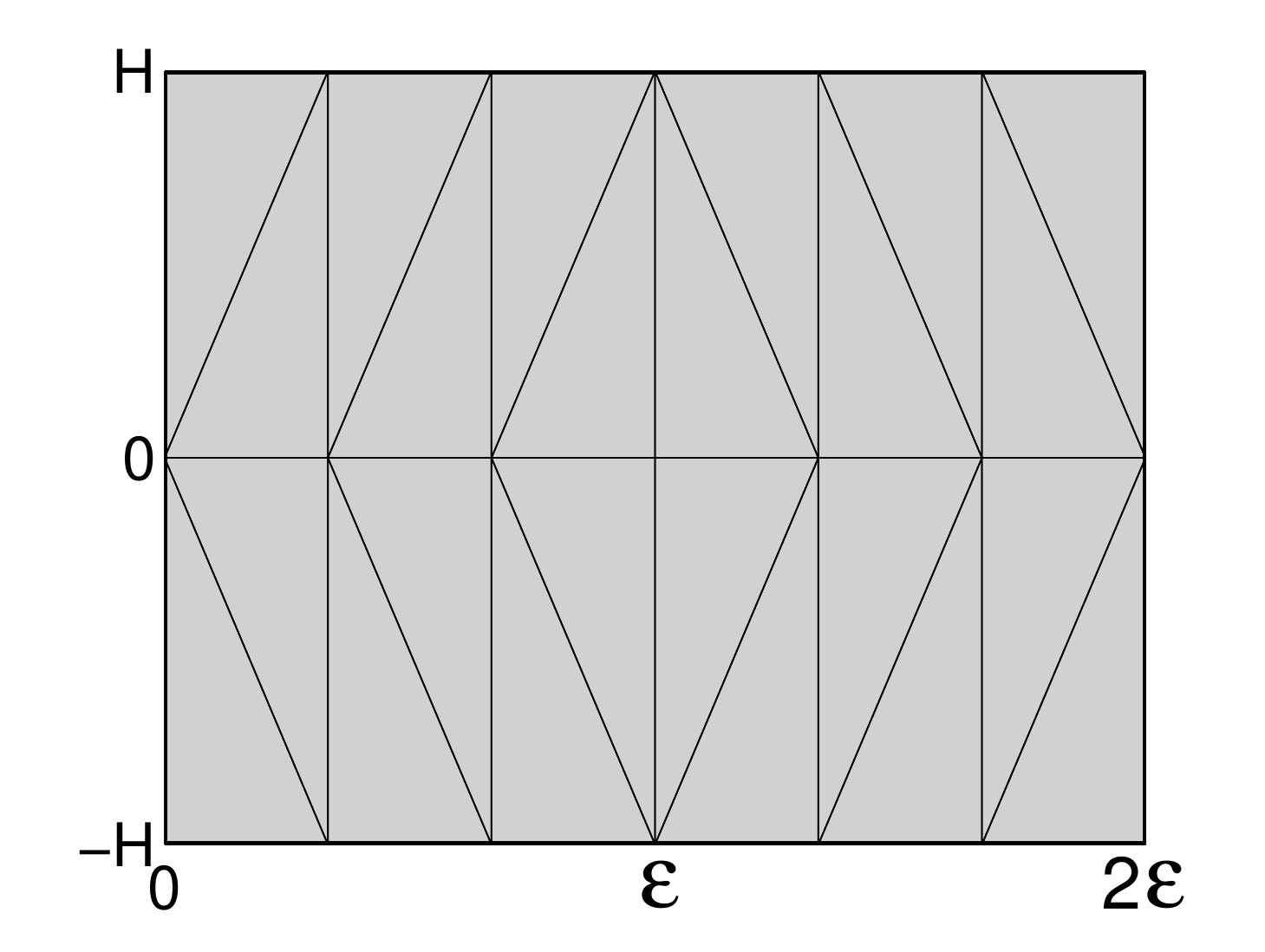}\hspace{0.5cm}\vspace{-0.2cm}
 \caption{\label{fig_Omega0}
Triangulation $\mathcal T$ in $\Omega$ (left); anisotropic subtriangulation $\mathring{\mathcal T}\subset\mathcal T$
of type C
is enlarged and stretched in the $x$-direction (right).}
 \end{figure}

 \subsection{Linear elements for singularly perturbed equation~\eqref{spp}}

The theoretical justification of the numerical phenomena of \S\ref{ssec_lin_spp} is addressed in \cite{Kopt_mc14}.
Here we briefly describe the  setting and the main results of  \cite[\S3]{Kopt_mc14},  as they will be useful in the forthcoming analysis.

Suppose $\Omega\supset\mathring{\Omega}$, where
 the subdomain $\mathring{\Omega}$ and the tensor-product grid $\mathring{\omega}_h$ in this subdomain are
 defined by
\beq\label{Omega_0}
 \mathring{\Omega}:=(0,2\eps)\times(-H,H),\quad
 \mathring{\omega}_h:=\{x_i=hi\}_{i=0}^{2N_0}\times\{-H,0,H\},\quad h={\ts\frac{\eps}{N_0}}.
\eeq
The triangulation $\mathring{\mathcal T}$ of type C  in $\mathring{\Omega}$ obtained from the rectangular grid $\mathring{\omega}_h$
is shown in Fig.\,\ref{fig_Omega0} (compare with  Fig.\,\ref{fig_ABC}, right).

 Next, let $U$
denote the linear finite element solution obtained on some triangulation ${\mathcal T}\supset\mathring{\mathcal T}$
 in the global domain $\Omega$, while
 for its nodal values in $\mathring\Omega$ we shall use the notation
 $$
 U_i:=U(x_i,0),\qquad U_i^{\pm}:=U(x_i,\pm H).
 $$
A calculation yields
 a finite difference representation
 in the {\it lumped-mass} case:%
 \begin{subequations}\label{fd_lumpted}
\beq\label{fd_lumpted_a}
\LL^h_{\rm l.m.} U(x_i,0):={\ts\frac{\eps^2}{h^2}}[-U_{i-1}+2U_i-U_{i+1}]
+{\ts\frac{\eps^2}{H^2}}[-U^-_{i}+2U_i-U_{i}^+]
+\gamma_i U_{i}=0\!\!
\eeq
 for $i=1,\ldots, 2N_0-1$, where\vspace{-0.2cm}
 \beq\label{fd_lumpted_b}
\gamma_i=1\quad\mbox{for}\;i\neq N_0,\qquad\quad  \gamma_{N_0}={\ts\frac23}.
 \eeq
 \end{subequations}
  If {\it no quadrature} is used, 
  for $i=1,\ldots, 2N_0-1$ one gets
\beq\label{fd_a}
\LL^h U(x_i,0):=\LL^h_{\rm l.m.} U(x_i,0)-{\ts \frac1{12}}\!\!\!\sum_{(x',y')\in {\mathcal S}_{i}}
\!\!\![U_{i}-U(x',y')]=0,
\vspace{-0.3cm}
\eeq
 where ${\mathcal S}_{i}$ denotes the set of mesh nodes that share an edge with $(x_i,0)$.%

 Note that if we had $\gamma_i= 1$ $\forall\, i$ , then  (\ref{fd_lumpted_a}) would become the standard 
 finite difference scheme for
  (\ref{spp}). 
 The difference in the value of $\gamma_i$ at
 a single point $(x_{N_0},0)$, 
where the truncation error is ${\mathcal O}(1)$,
 results in the deterioration of pointwise accuracy
 to ${\mathcal O}(N_0^{-1})$, as described in following lemma.

 \begin{lemma}[{\cite{Kopt_mc14}}]\label{lem_LM}
 Let $u=e^{-x/\eps}$ be the exact solution of 
 \eqref{spp} posed
 in  $\Omega\supset\mathring{\Omega}$,
 a triangulation $\mathcal T$ in $\Omega$
 include a
  subtriangulation $\mathring{\mathcal T}$ of type C in  $\mathring\Omega$
  subject to \eqref{Omega_0},
 and $U$ be the finite element solution on $\mathcal T$ obtained using  linear elements  without quadrature or with the lumped-mass quadrature.
For any positive constant $C_2$,
there exist sufficiently small constants $C_0$ and  $C_1$  such that
  if 
  $N_0^{-1}\le C_1$
  and
  $\eps\le C_2 H$,
  then\vspace{-0.2cm}
\beq\label{Uu_LM_bound}
 \max_{\bar\Omega}|U-u|  \ge 
C_0 N_0^{-1}.
\eeq
 \end{lemma}

Note that Lemma~\ref{lem_LM} applies to the two triangulations of type C considered in \S\ref{ssec_lin_spp}
 with $N_0=N/2$ for the uniform grid  in $(0,2\eps)\times(0,1)$, and $N_0={\mathcal O}(N)$,
 for the Bakhvalov grid in $(0,1)^2$.

 \subsection{Lumped-mass linear elements for the Laplace equation}\label{ssec_theory_Lapl}

Next, we address the numerical results of \S\ref{ssec_Lapl} for the Laplace equation.
The following lemma directly applies to the uniform anisotropic triangulation of type~C in the domain $(0,2\eps)\times(0,1)$ with $N_0=N/2$.
{\color{blue}Note that a version of this lemma can also be proved for the smooth mesh used in Table~\ref{Laplace_table_equi}, but the proof would be lengthier and more intricate.}

\begin{lemma}[Laplace equation]\label{lem_lapl}
Lemma~\ref{lem_LM} remains valid if the lumped-mass linear finite element method is applied to the Laplace equation with the exact solution $u=e^{-x/\eps}$.
\end{lemma}

 \begin{proof}
 A normalized version of \eqref{fd_lumpted_a} for this case becomes
$$
\LL^h_{\rm l.m.} U(x_i,0):={\ts\frac{\eps^2}{h^2}}[-U_{i-1}+2U_i-U_{i+1}]
+{\ts\frac{\eps^2}{H^2}}[-U^-_{i}+2U_i-U_{i}^+]
=-\gamma_i F_i,
$$
where $F_i:=\eps^2\triangle u(x_i,0)=e^{-x_i/\eps}$, while $\gamma_i$ remains defined by \eqref{fd_lumpted_b}.
Note that
$\gamma_i=1-\frac13\mathbbm{1}_{i=N_0}$ and $\gamma_i F_i=\eps^2\triangle u(x_i,0)-\frac13e^{-1}\mathbbm{1}_{i=N_0}$,
where,
{\color{blue}for any condition $A$,
 the indicator function $\mathbbm{1}_{\!A}$ is defined to be} equal to $1$ if condition $A$ is satisfied, and $0$ otherwise.%
\\[0.1cm]
 %
\indent (i) First, consider the auxiliary finite element solution $\mathring{U}$
 obtained on the triangulation $\mathring{\mathcal T}$ in $\mathring{\Omega}$,
 subject to the boundary condition
 $\mathring{U}=u^I$ on
  $\pt\mathring{\Omega}$.
 At the interior nodes, $\mathring{U}$ satisfies 
 $\LL^h_{\rm l.m.} \mathring{U}(x_i,0)=-\gamma_i F_i$  for $i=1,\ldots, 2N_0-1$.
 We shall now prove that, for a sufficiently small constant $C_0$,
 one has
\beq\label{U_0_bound}
[\mathring{U}-u](x_{N_0},0)  \ge 2C_0 N_0^{-1}.
\eeq

Let $e:=\mathring{U}-u$ and $e_i:=e(x_i,0)$.
As $e(x_i,\pm H)=0$, so $\LL^h_{\rm l.m.} e(x_i,0)$ can be rewritten in the form of
a one-dimensional discrete operator $L^h_x$ applied the vector $\{e_i\}_{i=0}^{2N_0}$ as
\beq\label{LLe}
\LL^h_{\rm l.m.} e(x_i,0)
=L^h_xe_i:= {\ts\frac{\eps^2}{h^2}}[-e_{i-1}+2e_i-e_{i+1}]
+{\ts\frac{2\eps^2}{H^2}} e_{i}\,.
\eeq
On the other hand, 
$\LL^h_{\rm l.m.} \mathring{U}(x_i,0)=-\gamma_i F_i=\frac13e^{-1}\mathbbm{1}_{i=N_0}-\eps^2\triangle u(x_i,0)$,
so the standard truncation error estimation yields
$
L^h_xe_i=\frac13e^{-1}\mathbbm{1}_{i=N_0}+{\mathcal O}\bigl({\ts\frac{ h^2}{\eps^2}}\bigr)$.

To simplify the presentation,  we shall complete the proof of \eqref{U_0_bound} under the condition $C_2\le 2^{-1/2}$, while a comment on the case $C_2>2^{-1/2}$ will be given at the end of this part of the proof.
Introduce the barrier function
$$
B_i:=\eps^{-1}\min\{x_i,\,2\eps-x_i\}-{\textstyle \frac12}\eps^{-2} x_i(2\eps-x_i).
$$
A calculation shows that $L^h_xB_i=\frac{2\eps}h\mathbbm{1}_{i=N_0}-1+\frac{2\eps^2}{H^2} B_i$.
Combining \mbox{$B_i\le 1$} with $ \frac{2\eps^2}{H^2}\le 2 C_2^2\le 1$, we arrive
at $L^h_xB_i\le \frac{2\eps}h\mathbbm{1}_{i=N_0}$.
As the operator $L^h_x$ satisfies the discrete maximum/comparison principle,
while $L^h_xe_i=\frac13e^{-1}\mathbbm{1}_{i=N_0}+{\mathcal O}\bigl({\ts\frac{ h^2}{\eps^2}}\bigr)$,
we conclude that
$e_i\ge \frac13e^{-1} \frac h{2\eps}B_i - {\mathcal O}\bigl({\ts\frac{ h^2}{\eps^2}}\bigr)$.
Noting that $B_{N_0}={\textstyle \frac12}$, while $h=\eps N_0^{-1}$, yields $e(x_{N_0},0)=e_{N_0}\ge\frac1{12}e^{-1}N_0^{-1}-{\mathcal O}(N_0^{-2})$.
This immediately implies \eqref{U_0_bound} for any constant $C_0<\frac1{24}e^{-1}$ assuming  $N_0$ is sufficiently large.

For the case $C_2> 2^{-1/2}$, one needs to slightly modify the above proof, replacing $B_i$ by
$\frac13e^{-1}h^{-1}G_i^h$, where $G_i^h$ is
the one-dimensional Green's function for the discrete operator $L_x^h$ associated with $i=N_0$; see \cite[proof of Lemma~3.3]{Kopt_mc14} for further details.
\smallskip

 (ii)
 In view of (\ref{U_0_bound}), to establish (\ref{Uu_LM_bound}), it suffices to show that
\beq\label{e0}
 \max_{\bar\Omega}|U-u^I|\ge {\ts\frac12} \max_{\mathring{\Omega}}|\mathring{U}-u^I|=: {\ts\frac12}\mathring{e},
 \vspace{-0.1cm}
 \eeq
 where $u^I$ is the standard piecewise-linear interpolant of $u$.
 Let $Z:=U-\mathring{U}$ in $\mathring{\Omega}$ and $Z_{\max}:=\sup_{\mathring{\Omega}} |Z|$.
 Note that
 $\LL^h_{\rm l.m.} Z(x_i,0)=0$, while
 $Z=U-u^I$ on $\pt\mathring{\Omega}$.
 As, by the discrete maximum principle, $|Z|$ attains its maximum in $\mathring{\Omega}$ on $\pt\mathring{\Omega}$, so
 $\max_{\pt\mathring{\Omega}} |U-u^I|=Z_{\max}$.
On the other hand, $U-u^I=(\mathring{U}-u^I)+Z$
 yields
$\max_{\mathring{\Omega}}|U-u^I|\ge \mathring{e}-Z_{\max}$ 
As the maximum of the two values $\mathring{e}-Z_{\max}$ and $Z_{\max}$ exceeds their average ${\ts\frac12}\mathring{e}$,
the desired relation (\ref{e0}) follows.
 \end{proof}

\subsection{Linear finite elements without quadrature on triangulations of type~B}\label{ssec_theory_quad}

Recall the numerical results of \S\ref{sec_quadr},
where
the singularly perturbed equation~\eqref{spp} was considered, and
linear finite elements without quadrature on triangulations of type~B exhibited only first-order pointwise accuracy.

Note that Lemma~\ref{lem_quadr} applies to the triangulations of type B considered in \S\ref{sec_quadr},
 with $N_0=N/2$ for the uniform grid  in $(0,2\eps)\times(0,1)$, and $N^{-1}_0={\mathcal O}(N^{-1}\ln N)$,
 for the Shishkin grid in $(0,1)^2$.

It was also observed in \S\ref{sec_quadr} that on the triangulations of type~B, the lumped-mass version
exhibits second-order convergence rates, superior compared to the version without quadrature. 
To understand this, note that
the  subtriangulation $\mathring{\mathcal T}$ in  $\mathring\Omega$ being now of type B
implies that the lumped-mass version can again be represented
as \eqref{fd_lumpted_a}, only with $\gamma_i=1$ $\forall\,i$.
The new definition of $\gamma_i$ is superior compared to $\gamma_i=1-\frac13\mathbbm{1}_{i=N_0}$, which one has for
similar triangulations of type C. Consequently, for the triangulations of type B, we no longer observe only first-order
accuracy described by Lemma~\ref{lem_LM}.

As to the version without quadrature, one again has \eqref{fd_a}, so
the additional terms $-\frac1{12}\sum [U_{i}-U(x',y')]$ in $\LL^h U(x_i,0)$ lead to
the following version of \eqref{LLe}:
\beq\label{LLe_quadr}
\LL^h e(x_i,0)
=L^h_xe_i:= \Bigl({\ts\frac{\eps^2}{h^2}}-{\ts \frac1{12}}\bigr)[-e_{i-1}+2e_i-e_{i+1}]
+\Bigl({\ts\frac{2\eps^2}{H^2}}+\underbrace{\gamma_i-{\ts \frac4{12}}}_{{}=\frac23\;\forall\,i}\Bigr) e_{i}\,.
\vspace{-0.1cm}
\eeq
Hence, because of the non-symmetric patch of elements touching each node $(x_i,0)$ in the version of $\mathring{\mathcal T}$ of type~B,
the truncation error $\LL^h e(x_i,0)=-\LL^h u(x_i,0)$
includes the two additional terms
${\ts \frac1{12}}[-u(x_{i-1},0)+2u(x_i,0)-u(x_{i+1},0)]={\mathcal O}\bigl({\ts\frac{ h^2}{\eps^2}}\bigr)$
and ${\ts \frac2{12}}[u(x_i,0)-u(x_{i+1},0)]=\frac{h}{6\eps} e^{-x_i/\eps}+{\mathcal O}\bigl({\ts\frac{ h^2}{\eps^2}}\bigr)$.
Finally,
\beq\label{LLe_truncation_}
\LL^h e(x_i,0)=\ts\frac{h}{6\eps} e^{-x_i/\eps}
+{\mathcal O}\bigl({\ts\frac{ h^2}{\eps^2}}\bigr),
\eeq
where the additional term $\frac{h}{6\eps} e^{-x_i/\eps}$ causes only first-order accuracy, described by the following lemma.

\begin{lemma}\label{lem_quadr}
If the  subtriangulation $\mathring{\mathcal T}$ in  $\mathring\Omega$ is of type B,
Lemma~\ref{lem_LM} remains valid for the linear finite element method without quadrature.
\end{lemma}

%
\begin{proof}
We imitate the argument used in the proof of Lemma~\ref{lem_lapl}; see also \cite[proof of Lemma~3.7]{Kopt_mc14}.

(i) To obtain~\eqref{U_0_bound}, we combine \eqref{LLe_quadr} with \eqref{LLe_truncation_}, and then employ
the barrier function $B_i:=B(x_i/\eps)$, where 
$(-\frac{d^2}{dt^2}+\frac23+C_2^2)B(t)=e^{-t}$ subject to $B(0)=B(2)=0$.
Clearly, $B$ is smooth and positive on $(0,2)$.
Now, again using the discrete maximum principle, one can show that
$e(x_i,0)\ge\frac{h}{6\eps}B(x_i/\eps)-{\mathcal O}\bigl({\ts\frac{ h^2}{\eps^2}}\bigr)$, which implies \eqref{U_0_bound}.

(ii) This part of the proof is as in the proof of Lemma~\ref{lem_lapl}, except
we need to be more careful when claiming that  $|Z|$ 
attains its  maximum  in $\mathring{\Omega}$ on $\pt\mathring{\Omega}$.
Note that, in view of \eqref{fd_a}, one can rewrite $\LL^h  Z_i=0$ using $L_x^h$ from \eqref{LLe_quadr} as
$$
L_x^hZ_i = -{\ts \frac1{12}}\bigl[Z(x_{i},H)+Z(x_{i},-H)+Z(x_{i+1},H)+Z(x_{i+1},-H)\bigr].
$$
So $|L_x^hZ_i|\le {\ts \frac1{3}}\max_{\pt\mathring{\Omega}}|Z|$.
{\color{blue}Now, in view of $L_x^h[1]\ge \frac23$, using the discrete maximum principle, we conclude that
$\max|Z_i|\le \max\bigl\{ |Z_0|,\, |Z_{2N_0}|,\,\frac32 \max|L_x^hZ_i| \bigr\}$.
Combining this with the bound on $|L_x^hZ_i|$ yields
 $\max|Z_i|\le \max_{\pt\mathring{\Omega}}|Z|$, so, indeed, $|Z|$ attains its maximum in $\mathring{\Omega}$ on $\pt\mathring{\Omega}$.}
\end{proof}

{\color{blue}

\section{Tested meshes in view of Hessian-based metrics}\label{sec_hessian}
Note that \eqref{ScjWal} is frequently considered a reasonable heuristic conjecture
to be used in the anisotropic mesh adaptation. In particular, when second-order methods are employed,
mesh generators frequently aim to produce meshes that are quasi-uniform under
Hessian-based metrics.

To be more precise,
given $u\in C^2(\bar\Omega)$ with its diagonalized Hessian $D^2 u=Q^t{\rm diag}\bigl(\lambda_i\bigr)Q$,
such a metric may be induced by the matrix
$\HH:=Q^t{\rm diag}\bigl(|\lambda_i|\bigr)Q+\theta I$, where the presence of the identity matrix $I$ multiplied by a constant $\theta\ge 0$ ensures that $\HH$ is positive-definite.
Alternatively, one can employ a similar
$\HH:=Q^t{\rm diag}\bigl(\max\{|\lambda_i|,\,\theta\}\bigr)Q$.

Let us look at the considered meshes in view of such metrics. Note that the considered exact solutions $e^{-x/\eps}$ and $x^{1/2}$
have singular Hessian matrices, so
reqire $\theta>0$.

1. In view of Remark~\ref{rem_Bakh}, the considered Bakhvalov mesh is quasi-uniform under the above metric with $\theta =O(1)$
(note that the mesh is almost uniform outside the layer region, with the mesh size close to $M^{-1}=4N^{-1}$).
Similarly, the uniform mesh, used in~\S\ref{sec_linear} in the domain $(0,2\eps)\times(0,1)$, is quasi-uniform under the above metric with $\theta =O(1)$.

2. In \S\ref{ssec_Lapl} we also use the mesh that is uniform under the 1d Hessian metric in the $x$-direction, with $M=\frac14N$ and $M=16$.
The resulting 2d meshes are close to uniform under the Hessian metric with, respectively,  $\theta=O(1)$ and $\theta\ll 1$.
Note that the latter choice corresponds to
 more anisotropic meshes and, on triangulations of type C,
 convergence rates becoming close to $1$ even for $\eps=1$ (see Table~\ref{Laplace_table_equi}).

 3. For the graded mesh in the $x$-direction, used for a singular problem in \S\ref{ssec_singular},
 a calculation (imitating the argument in Remark~\ref{rem_Bakh}) again shows that it is uniform under the 1d Hessian metric (with the obvious exception of the first mesh interval).

 4. The theoretical results of \S\ref{ssec_theory_quad}
 apply to
 the exact solution $u=e^{-x/\eps}$ and
 an arbitrary triangulation of $\Omega$, subject to certain conditions in the subdomain $\mathring{\Omega}$, including \eqref{Omega_0}.
In $\mathring{\Omega}$, the mesh $\{x_i\}$ in  the $x$-direction is quasi-uniform under the 1d Hessian metric, so
 setting the mesh size in the $y$-direction $H=O(\theta^{-1}N_0^{-1})$ produces
 a 2d mesh in
  $\mathring{\Omega}$ that is quasi-uniform under the 2d Hessian metric.
  Not only our theoretical results apply to arbitrarily small $\theta$, but if
  $\theta$ is sufficiently small, they apply even to the case $\eps=1$ (which is consistent with the lower part of Table~\ref{Laplace_table_equi}).}

\end{document}